\documentclass[12pt]{amsart}
\usepackage[dvips]{graphicx}
\usepackage{amsmath,graphics,xypic}
\usepackage{color}
\usepackage{amsfonts,amssymb,hyperref}
\usepackage{enumitem}
\usepackage{tikz-cd}
\usepackage{marvosym}
\usepackage{hyperref}

\theoremstyle{plain}
\newtheorem{theorem*}{Theorem}
\newtheorem*{lemma*}{Lemma}
\newtheorem{corollary*}{Corollary}
\newtheorem*{proposition*}{Proposition}
\newtheorem{conjecture*}{Conjecture}
\newtheorem{theorem}{Theorem}[section]
\newtheorem{lemma}[theorem]{Lemma}

\newtheorem{corollary}[theorem]{Corollary}

\newtheorem{proposition}[theorem]{Proposition}
\newcounter{lem}
\newtheorem{alphalemma}[lem]{Lemma}
\theoremstyle{remark}

\newtheorem*{remark}{Remark}

\newtheorem*{definition}{Definition}

\newtheorem*{claim}{Claim}

\def\LAT{\tau^{(2)}}

\def\op{\operatorname}

  \def\Z{\Bbb{Z}} \def\R{\Bbb{R}} 
 \def\genus{\op{genus}}   
   \def\bp{\begin{pmatrix}}
	 \def\ep{\end{pmatrix}} \def\bn{\begin{enumerate}} 
	   \def\en{\end{enumerate}}
\def\ba{\begin{array}} \def\ea{\end{array}}  
\def\id{\op{id}}  
\def\hom{\op{Hom}} 
\def\dim{\operatorname{dim}}

\begin{document}
	\author{Gerrit Herrmann}
	\address {Fakult\"at f\"ur Mathematik \\
		Universit\"at Regensburg\\
		Germany\\
		\newline\url{www.gerrit-herrmann.de} }
	\email{gerrit.herrmann@mathematik.uni-regensburg.de}
	
	\title{The $L^2$-Alexander torsion for Seifert fiber spaces}
	
	\begin{abstract}
We calculate the $L^2$-Alexander torsion for Seifert fiber spaces and graph manifolds in terms of the Thurston norm.
	\end{abstract}
	\maketitle
	\section{Introduction and main result}
	We say two functions $f,g\colon \R_{>0}\rightarrow \R_{\geq 0}$ are equivalent and write $f\doteq g$, if there exists an $r\in \R$ such that $f(t)=t^{r} g(t)$ for all $t\in\R_{>0}$.
In \cite{DFL14a} Dubois, Friedl and L\"uck associated to an admissible triple $(M,\phi,\gamma)$ a equivalence class of a function $\LAT (M,\phi,\gamma)\colon \R_{> 0} \rightarrow \R_{\geq 0}$ called $L^2$-Alexander torsion. The triple consists of a manifold $M$, a cohomology class $\phi \in H^1(M;\R)$ and a group homomorphism $\gamma\colon \pi_1(M)\rightarrow G$ such that $\phi$ factors through $\gamma$.

The $L^2$-Alexander torsion is a generalization of the $L^2$-Alexander polynomial introduced by Li-Zhang \cite{LZ06a}, \cite{LZ06b}, \cite{LZ08} and has been studied recently by many authors Dubois-Wegner \cite{DW10}, \cite{DW13}, Ben Aribi \cite{BA13}, \cite{BA16}
Dubois-Friedl-L\"uck \cite{DFL14a}, \cite{DFL14b}, \cite{DFL15}, Friedl-L\"uck \cite{FL15} and Liu \cite{Li15}.

 Given a 2-dimensional manifold $S$ with connected components $S_1\cup\ldots\cup S_k$ we define its complexity to be $\chi_-(S):=\sum_{i=1}^{k}\max \left\{-\chi(S_i),0\right\} $. Let $M$ be 3-manifold (throughout this paper every 3-manifold is understood to be orientable, compact and have empty or toroidal boundary). Thurston has shown in \cite{Th86} that the function  
\begin{align*}
x_M\colon H^1(M;\Z)&\longrightarrow \Z\\
\phi&\longmapsto \min\left\{\chi_-(S)\ \left|\ \begin{array}{ll}
[S] \text{ is Poincar\'{e} dual to $\phi$} \\ 
\text{ and properly embedded}
\end{array}\right.\right\}. 
\end{align*}
extends to a semi-norm on $H^1(M;\R)$.

 In this paper we will show, that for a Seifert fiber space $M$ the function $\LAT (M,\phi,\gamma)$ is completely determined by $x_M(\phi)$. To be more precise:
%
%

	\begin{theorem}[Main Theorem]\label{maintheroem}
		Let $M$ be a Seifert fiber space with $M\neq S^1\times S^2,S^1\times D^2$ and $(M,\phi,\gamma)$ an admissible triple, such that the image of a regular fiber under $\gamma$ is an element with infinite order.
		Then the  $L^2$-Alexander torsion is given by
		\[ \tau^{(2)}(M,\phi,\gamma) \doteq \max \left\{1, t^{x_M(\phi)}\right\}.\]
	\end{theorem}
Theorem \ref{maintheroem} was already used in \cite{DFL14a} to show the following two corollaries. 
\begin{corollary}\label{cor:graph}
	Let $(M,\phi,\gamma)$ be an admissible triple with $M\neq S^1\times S^2,S^1\times D^2$. Suppose that $M$ is a graph manifold and that given any $JSJ$-component of $M$ the image of a regular fiber under $\gamma$ is an element of infinite order, then
	\[\LAT ( M,\phi,\gamma)=\max\left\{1,t^{x_M(\phi)}\right\}. \]
\end{corollary}

Note that the next corollary was first proved by Ben Aribi \cite{BA13} using a somewhat different approach.
\begin{corollary}
Let $K\subset S^3$ be a oriented knot. Denote by $\nu K$ a tubular neighborhood of $K$ and by $\phi_K\in H^1(S^3\setminus\nu K)$ the cohomology class which sends the oriented meridian to $1$. Then $K$  is trivial if and only if $\LAT(S^3\setminus\nu K,\phi, \id )(t)\doteq \max\left\{1,t\right\}^{-1}$.
\end{corollary}
\begin{proof}
	If the JSJ-decomposition of $S^3\setminus \nu K$ contains a hyperbolic piece, then by the work of L\"uck and Schick \cite{LS99} we have $\LAT(S^3\setminus \nu K, \phi,\id)(1)\neq 1$. If there is no hyperbolic piece, then $S^3\setminus \nu K$ is a graph manifold. Now the statement follows from Corollary \ref{cor:graph} and the well known formula $x_{S^3\setminus \nu K}(\phi_K)=2\cdot \genus (K) -1$, if $\genus(K)>0$.
\end{proof}
The proof of the main theorem is obtained from two lemmas. The first lemma characterizes the Thurston norm of a Seifert fiber space $M$ by the combinatorial invariant $\chi^{S^1}_{\text{orb}}(M)$ (see Definition \ref{def:orbchar}).
\begin{alphalemma}
	Let $M\neq S^1\times S^2,\ S^1\times D^2$ be a Seifert fibered space with infinite fundamental group. Then for any $\phi\in H^1(M;\R)$, we have 
	\begin{align*}
	x_M(\phi) = |\chi^{S^1}_{\op{orb}}(M)\cdot k_\phi |,
	\end{align*}
	where $k_\phi:= \phi ([F])$ and $F$ is a regular fiber.
\end{alphalemma}
The second lemma calculates the function $\LAT(X,\phi,\gamma)$ for spaces with a certain $S^1$-action.
\begin{alphalemma}\label{l2torofs1cwc}
		Let X be a connected $S^1$-CW-complex of finite type and $\phi\in H^1(X;\R)$.
		Suppose that for one and hence all $x\in X$ the map $ev_x\colon S^1\rightarrow X$ defined by $z\mapsto z\cdot x$ induces an injective map $\gamma\circ ev_x\colon \pi_1(S^1,1)\rightarrow G$. The composite
		\[\begin{tikzcd}
		\pi_1(S^1,1)\arrow{r}{ev_x}& \pi_1(X,x)\arrow{r}{\phi}&\R
		\end{tikzcd}\]
		is given by multiplication with a real number. Let $k_\phi$ denote this number.
		Define the $S^1$-orbifold Euler characteristic of X by 
		\[\chi_{\op{orb}}^{S^1} (X) = \sum_{n\geq 0} (-1)^n \cdot \sum_{i\in J_n}\frac{1}{|H_i|}\]
		where $J_n$ denotes the set of open $n$-cells and for $i\in J_n$ the set $H_i$ is the isotropy group of the corresponding cell. Then
		\[\tau^{(2)}(X,\phi,\gamma) \doteq \max \left\{1, t^{k_\phi}\right\}^{-\chi_{\op{orb}}^{S^1} (X)}.\]
\end{alphalemma}
The paper is organized as follows. In Section \ref{thurstonnormforseifertfiberspaces} we recall the definition of Seifert fiber spaces and prove Lemma \ref{thurstonnormforallSFS}. Afterwards we give in Section \ref{sec:basicLAT} the definition of the $L^2$-Alexander torsion and some basic properties. In the last part of the paper we prove Lemma \ref{l2torofs1cwc}. Note that not all Seifert fiber spaces admit an $S^1$-action and so we will prove the main result for the remaining cases in the remainder of this paper.
\subsection{Acknowledgements} I would like to thank my advisor Stefan Friedl for his many suggestions and his helpful advice. The article is based on work supported by SFB 1085 at the University of Regensburg, funded by the Deutsche Forschungsgemeinschaft (DFG). Moreover I wish to thank Fathi Ben Aribi for several helpful conversations.

\section{Thurston norm for Seifert fiber spaces}\label{thurstonnormforseifertfiberspaces}
\subsection{Preliminaries about Seifert fiber spaces}
We quickly recall the definition and basic facts about Seifert fiber spaces. Most results are taken from the survey article \cite{Sc83}.

\begin{definition}
A \emph{Seifert fibered space} is a 3-manifold $M$ together with a decomposition of $M$ into disjoint simple closed curves (called Seifert fibers) such that each Seifert fiber has a tubular neighborhood that forms a standard fibered torus. The standard fibered torus corresponding to a pair of coprime integers $(a;b)$ with $a > 0$ is the
surface bundle of the automorphism of a disk given by rotation by an angle of $2\pi b/a$,
equipped with the natural fibering by circles.

We call $a$ the index of a Seifert fiber. A fiber is \emph{regular} if the index is one and \emph{exceptional} otherwise.\end{definition}
\begin{remark}
	The number of exceptional fibers of a Seifert fiber space $M$ is finite.
\end{remark}

\begin{definition}\label{def:orbchar}
Let $M$ be a Seifert fibered space and $F_1,\ldots,F_n$ the exceptional fibers with corresponding indices  $a_1,\ldots,a_n$. We define
\[ \chi^{S^1}_{\op{orb}}(M):=\chi(M/S^1)- \sum_{i=1}^{n} \Big(1-\frac{1}{a_i}\Big),\]
where $M/S^1$ is obtained from $M$ by identifying all points in the same Seifert fiber.
\end{definition}
\begin{remark}
	In Scott's survey article this is the Euler number of the base orbifold $M/ S^1$.
\end{remark} 
Let $p\colon\widehat{M}\rightarrow M$ be a finite cover. If $M$ is Seifert fibered then $p$ induces a Seifert fiber structure on $\widehat{M}$ such that $p$ is a fiber preserving map. Therefore we get a induced branched cover $p\colon\widehat{M}/S^1\rightarrow M/S^1$. Denote by $l$ the degree of the branched cover. Then standard arguments show
\begin{align*}
l\cdot \chi^{S^1}_{\op{orb}}(M)= \chi^{S^1}_{\op{orb}}(\widehat{M}).
\end{align*}
For more detail we refer to \cite[Section 2 and 3]{Sc83}.

\subsection{Proof of Lemma \ref{thurstonnormforallSFS}}
Having all the notions to our hand we can prove:
\setcounter{lem}{0}
\begin{alphalemma}\label{thurstonnormforallSFS}
	Let $M\neq S^1\times S^2,\ S^1\times D^2$ be a Seifert fibered space with infinite fundamental group. Then for any $\phi\in H^1(M;\R)$, we have 
	\begin{align*}
	x_M(\phi) = |\chi^{S^1}_{\op{orb}}(M)\cdot k_\phi |,
	\end{align*}
	where $k_\phi:= \phi ([F])$ and $F$ is a regular fiber.
\end{alphalemma}

The proof will consist of the following steps. First we reduce the problem from Seifert fiber spaces to principal $S^1$-bundles. Then we look at the two cases of a trivial and non trivial principal $S^1$-bundle separately.
\begin{claim}
	It is sufficient to prove Lemma \ref{thurstonnormforallSFS} only for principal $S^1$-bundles.
\end{claim} 
\begin{proof}
	As shown in \cite[Section 3.2(C.10)]{AFW15} a Seifert fiber space is finitely covered by a principal $S^1$-bundle.  Let $p\colon\widehat{M}\rightarrow M$ be such a finite cover. We can pullback the Seifert fiber structure from $M$ to $\widehat{M}$. This structure and the structure of the $S^1$-bundle on $\widehat{M}$ coincide, because the Seifert fiber structure of an aspherical $S^1$-bundle is unique by the argument of \cite[Theorem 3.8]{Sc83}. Let $m$ denote the degree with which regular fibers of $\widehat{M}$ cover regular fibers of $M$. Then we have $m\cdot k_\phi=k_{p^{\ast}\phi}$. Denote by $l$ the degree of the induced branched cover $p\colon\widehat{M}/S^1\rightarrow M/S^1$. Then $l\cdot m$ is the degree of the cover $p\colon\widehat{M}\rightarrow M$. From the discussion above we deduce \[\chi^{S^1}_{orb}(M)\cdot k_{\phi} =  \frac{\chi^{S^1}_{orb}(\widehat{M}) \cdot k_{p^{\ast}\phi}}{d}.\]
	Furthermore Gabai showed in \cite[Corollary 6.13]{Ga83} for a finite cover of 3-manifold $p\colon\widehat{M}\rightarrow M$:
	\[x_M(\phi)=\frac{x_{\widehat{M}}(p^{\ast}\phi)}{d}.\]
	Putting together both equations, we see that it is enough to prove Lemma \ref{thurstonnormforallSFS} for $S^1$-bundles.
\end{proof}
For the next proof we need the following well known fact about fiber bundles and the Thurston norm. If $\Sigma_g\rightarrow M \xrightarrow{p} S^1$ is a fiber bundle and $\phi$ is given by $p_\ast\colon H_1(M;\R)\rightarrow H_1(S^1;\R)$, then 
 \begin{align}\label{thurstonnormfiberbundles}
  x_M(\phi) = \begin{cases}
 -\chi(\Sigma_g)&\text{if } \chi(\Sigma_g)<0, \\
 0&\text{else}.
 \end{cases}
 \end{align} 
\begin{lemma}\label{lem:trivialbundle}
	Let $\Sigma_g$ be a surface with $\chi(\Sigma_g)<0$. Consider $M=S^1\times \Sigma_g$ and $\phi\in H^1(M;\R)$. Then we have: 
		\[x_{M}(\phi)=|k_\phi\cdot \chi(\Sigma_g) |,\]
		where $k_\phi=\phi([F])$ and $F$ is regular fiber.
\end{lemma}
\begin{proof}In the following homology and cohomology is understood with real coefficients.
	By the K\"unneth theorem $H^1(M)\cong H_2(M,\partial M)$ decomposes into:
	\[ H_2(M, \partial M) \cong H_2(\Sigma_g,\partial \Sigma_g)\  \oplus\  H_1(\Sigma_g,\partial \Sigma_g)\otimes H_1(S^1)\]
	Every generator of $H_1(\Sigma_g,\partial \Sigma_g)\otimes H_1(S^1)$ can be represented by a torus and therefore $H_1(\Sigma_g,\partial \Sigma_g)\otimes H_1(S^1)$ is a subspace with vanishing Thurston norm. The number $k_\phi$ can be interpreted as the intersection number of a regular fiber $F$ with the surface representing the cohomology class $\phi$. A surface representing $\phi\in H_1(\Sigma_g,\partial \Sigma_g)\otimes H_1(S^1)$ is parallel to a regular fiber and therefore $k_\phi=0$.
	
	Since $H_1(\Sigma_g,\partial \Sigma_g)\otimes H_1(S^1)$ is a subspace of vanishing Thurston norm, the Thurston norm of a class does not change, if we add elements of $H_1(\Sigma_g,\partial \Sigma_g)\otimes H_1(S^1)$.  
	Moreover $k_\phi$ is linear in $\phi$ and
	therefore the last open case is $PD(\phi)=[\Sigma_g]\in H_2(M,\partial M)$. This is equivalent with $\phi$ being the fiber class of the fibration $\Sigma_g \rightarrow S^1\times \Sigma_g \rightarrow S^1$ and there the formula holds by Equation (\ref{thurstonnormfiberbundles}) and the fact, that in this case $k_\phi=1$.
	
\end{proof}
\begin{lemma}
	Let $M$ be a non trivial $S^1$-bundle over a surface $\Sigma_g$. Then the following equation 
		\[x_{M}(\phi)=| k_\phi\cdot \chi(\Sigma_g) |\]
	holds for all $\phi\in H^1(M;\R)$.
\end{lemma}
\begin{proof}
	As in the proof of Lemma \ref{lem:trivialbundle} homology and cohomology is understood with real coefficients.
	We will in fact proof that both sides are equal to zero. To calculate $H^1(M)$ we look at a part of the Gysin short exact sequence \cite[Theorem 13.2]{Br93}:
		\[
		\begin{tikzcd}
		0 \arrow{r}& H^1(\Sigma_g)\arrow{r}{p^\ast}& H^1(M) \arrow{r}& H^0(\Sigma_g)\arrow{r}{\cup\, e}& H^2(\Sigma_g)\arrow{r}&\ldots 
		\end{tikzcd}
		\]
Here $\cup\, e$ means the cup product with the Euler class $e$ associated to the $S^1$-bundle $M$. This is an isomorphism because the bundle is not trivial. Therefore $p^{\ast}$ is an isomorphism too. We apply Poincar\'e duality and get an isomorphism $p_{!}:H_1(\Sigma_g)\rightarrow H_2(M)$. By the argument of \cite[Section 4]{GH14} an element $c\in H_1(\Sigma_g;\Z)$ will be sent by $p_{!}$ to a class in $H_2(M)$ representable by the product of a regular fiber with curves in  $\Sigma_g$ representing $c$. This is a collection of tori and annuli. We conclude that all elements in $H_2(M)$ have trivial Thurston norm. We have a second look at the Gysin sequence to show that $0=[F]\in H_1(M)$ for a regular fiber $F$.
\[
\begin{tikzcd}
\ldots\arrow{r}& H^0(\Sigma_g)\arrow{r} {\cup\, e}& H^2(\Sigma_g)\arrow{r}{p^{\ast}}& H^2(M) \arrow{r} &\ldots
\end{tikzcd}
\]
We change again via Poincar\'e duality to homology and obtain the map  $p_{!}\colon H_0(\Sigma_g)\rightarrow H_1(M)$, $[x_0]\mapsto [F]$ which sends a point to a regular fiber $F$. This map is trivial, because $\cup\, e$ is surjective. Hence $[F]=0$ and we conclude $k_\phi=\phi([F])=0$.

\end{proof}
\section{Definition and basic properties of the $L^2$-Alexander torsion}\label{sec:basicLAT}
In the monograph \cite{Lu02} L\"uck defines the $L^2$-torsion $\rho^{(2)}(C_\ast)\!\in\!\R$ of a finite Hilbert-$\mathcal{N}(G)$ chain complex of determinant class $C_\ast$. We refer to \cite[Chapter 3]{Lu02} for the precise definition and basic properties.
\begin{definition}
	Let $X$ be a connected finite CW-complex, $\phi\in H^{1}(X,\R)=\hom(\pi_1(X),\R)$, and $\gamma\colon  \pi_1(X)\rightarrow G$ a group homomorphism.  We call $(X,\phi,\gamma)$ an \emph{admissible triple} if $\phi \colon \pi_1(X)\rightarrow\R$ factors through $\gamma$ i.e.\ there exists $\phi'\colon G\rightarrow\R$ such that $\phi=\phi'\gamma$.
\end{definition}
An admissible triple $(X,\phi,\gamma)$ and positive number $t\in \R_{>0}$ give rise to a ring homomorphism:
\begin{align*}
\kappa(\phi,\gamma,t)\colon\R[\pi_1(X)] &\longrightarrow \R[G] \\
\sum_{i=1}^{n}a_i g_i&\longmapsto \sum_{i=1}^{n}a_it^{\phi(g_i)} \gamma (g_i).
\end{align*}

\begin{definition}[$L^2$-Alexander torsion]
	Let $(X,\phi,\gamma)$ be an admissible triple. We write $C_\ast^{\phi,\gamma,t}:= l^2(G)\otimes_{\kappa(\phi,\gamma,t)} C_\ast (\widetilde{X})$, where $\widetilde{X}$ is the universal cover and consider the function
	\[\LAT(X,\phi,\gamma)(t):= \begin{cases} \exp\left(-\rho^{(2)} \big(C^{\phi,\gamma,t}_\ast \big)\right) & C^{\phi,\gamma,t}_\ast \text{ {\footnotesize  is of determinant class and weakly acyclic,}} \\
	0 &\text{else.}
	\end{cases}\]
This function may not be continuous in general, but Liu showed  that in the case of a 3-manifold and $\gamma=\id$ this function is always greater than zero and continuous \cite[Theorem 1.2]{Li15}.
\end{definition}
We can define the $L^2$-Alexander torsion for a pair of spaces $(X,Y)$ in the following way. Let $(X,\phi,\gamma)$ be an admissible triple and $Y\subset X$ a subcomplex. We denote by $p:\widetilde{X}\rightarrow X$ the universal cover of $X$. We write $\widetilde{Y}=p^{-1}(Y)$. Then  $C_\ast(\widetilde{X},\widetilde{Y})$ is a free left $\Z[\pi_1(\widetilde{X})]$-chain complex. We write $C_\ast^{\phi,\gamma,t}:= l^2(G)\otimes_{\kappa(\phi,\gamma,t)} C_\ast (\widetilde{X},\widetilde{Y})$ and define the $L^2$-Alexander torsion $\LAT(X,Y,\phi,\gamma)(t)$ as before.

We will make use of the following two proposition. The proofs are straight forward. One applies \cite[Theorem 3.35(1)]{Lu02} to the short exact sequences of the chain complexes in question.
 \begin{proposition}[Product formula]\label{torsionofpair}
 	Let $(X,\phi,\gamma)$ be an admissible triple and  $i\colon Y\hookrightarrow X$ a subcomplex. If two out of three of $\LAT({X,\phi,\gamma})$, $\LAT({X,\phi i_\ast,\gamma i_\ast})$ and $\LAT({X,Y,\phi,\gamma})$ are nonzero, then we have the identity:
 	\[\LAT({Y,\phi i_{\ast},\gamma i_{\ast}})\cdot \LAT({X,Y,\phi,\gamma})\doteq \LAT({X,\phi,\gamma})
 	 .\]
 	
 \end{proposition}

\begin{proposition}[Gluing formula]\label{pro:LATgluing}
	Consider a pushout diagram of finite CW-complexes
	\[\begin{tikzcd}
	X_0 \arrow{r}{i_1} \arrow{d}{i_2}&X_{1}\arrow{d}{j_2} \\
	X_2 \arrow{r}{j_2}& X_3
	\end{tikzcd}\]
	 such that every map is cellular and $i_1$ is injective. Let $(X_3,\phi,\gamma)$ be an admissible triple. If three out of four of $\LAT(X_0,\phi i_1 j_1, \gamma i_1 j_1),\,\LAT(X_1,\phi j_1,\gamma j_1),\,\LAT(X_2,\phi j_2,\gamma j_2)$ or $\LAT(X_3,\phi,\gamma)$ are nonzero, then we have:
	 \[\LAT(X_3,\phi,\gamma)\cdot \LAT(X_0,\phi i_1 j_1, \gamma i_1 j_1)\doteq \LAT(X_2,\phi j_2,\gamma j_2)\cdot\LAT(X_1,\phi j_1,\gamma j_1).  \]
\end{proposition}
\subsection{The $L^2$-Alexander torsion for $S^1$-CW-complexes}\label{l2alextorfors1cwcomplex}
The calculation of the $L^2$-Alexander torsion of a circle can be found in \cite[Lemma 2.8]{DFL14a}. It is the starting point of the proof of Lemma \ref{l2torofs1cwc}.
\begin{lemma}\label{l2torsionofacircle}
	Let $(S^1,\phi,\gamma)$ be an admissible triple and $\gamma$ injective, then the $L^2$-Alexander torsion is given by
	\[ \LAT(S^1,\phi,\gamma)\doteq \max\left\{1,t^{\phi(g)}\right\}^{-1},\]
	where $g$ is a generator of $\pi_1(S^1)$.
\end{lemma} 
\begin{alphalemma}
	Let X be a connected $S^1$-CW-complex of finite type and $\phi\in H^1(X;\R)$.
	Suppose that for one and hence all $x\in X$ the map $ev_x\colon S^1\rightarrow X$ defined by $z\mapsto z\cdot x$ induces an injective map $\gamma\circ ev_x\colon \pi_1(S^1,1)\rightarrow G$. The map $\phi\circ ev_x\colon H_1(S^1)\rightarrow \R$
	is given by multiplication with a real number which we denote by $k_\phi$. We obtain:
	\[\tau^{2}(X,\phi,\gamma)\doteq \max \left\{1, t^{k_\phi}\right\}^{- \chi_{\op{orb}}^{S^1} (X)}\]
\end{alphalemma}
This proof is a variation of the proof \cite[Theorem 3.105]{Lu02}:
\begin{proof}
	We use induction over the dimension of cells. In dimension zero $X$ is a circle. There the statement holds by Lemma \ref{l2torsionofacircle}. Now the induction step from $n-1$ to $n$ is done as follows.
	Per definition of an $S^1$-CW-complex we can choose an equivariant $S^1$-pushout with $\dim (X_{n})=n$
	\[\begin{tikzcd}
	\bigsqcup_{i\in J_n} S/H_i\times S^{n-1} \arrow{r}{\bigsqcup q_i}\arrow{d}{i}&X_{n-1}\arrow{d}{j} \\
	\bigsqcup_{i\in J_n} S/H_i\times D^{n} \arrow{r}{\bigsqcup Q_i}& X_n\ .
	\end{tikzcd}\]
	We obtain from the gluing formula \ref{pro:LATgluing}:
	\[\prod_{i\in J_n}\LAT ({\displaystyle S^1/H_i\!\times\! D^n,S^1/H_i\!\times\! S^{n-1},\phi Q_i ,\gamma Q_i})\doteq \LAT(X_n,X_{n-1},\phi, \gamma).\]
	The left hand side can be computed by the suspension isomorphism.
	
	\begin{align*}
	\prod_{i\in J_n} \LAT({\displaystyle S^1/H_i\!\times\! D^n,S^1/H_i\!\times\! S^{n-1},\phi Q_i ,\gamma Q_i}) &\doteq \prod_{i\in J_n} \LAT({\widetilde{S^1},\phi Q_i,\gamma Q_i}) ^{(-1)^{n} }  \\ &\doteq  \prod_{i\in J_n}\max\left\{1,t^{k_\phi}\right\}^{(-1)^{n+1}\cdot 1/|H_i|} \\
	&=\max \left\{1,t^{k_\phi}\right\}^{(-1)^{n+1}\sum_{i\in J_n}1/|H_i|}. 
	\end{align*}
	Here we used the assumption that $\gamma\circ ev_x$ is injective. The torsion for $X_{n-1}$ is defined by induction hypothesis. We can apply the product formula \ref{torsionofpair} and conclude:
	\begin{align*}
	\LAT({X_n,\phi, \gamma})&=\LAT({X_n,X_{n-1}\phi, \gamma})\cdot \LAT (X_{n-1},\phi i, \gamma i) \\
	&\doteq \max \left\{1,t^{k_\phi}\right\}^{(-1)^{n+1}\sum_{i\in J_n}1/|H_i|} \cdot \max \left\{1,t^{k_\phi } \right\}^{-\chi^{S^1}_{\op{orb}}(X_{n-1})} \\
	&= \max \left\{1, t^{k_\phi}\right\}^{-\chi_{\op{orb}}^{S^1} (X)}.
	\end{align*}
	
\end{proof}

\begin{corollary}\label{cor:Mainthm}
	Let $M$ be an aspherical Seifert fiber space with $M\neq S^1\times D^2$ and $(M,\phi,\gamma)$ an admissible triple, such that the image of a regular fiber under $\gamma$ is an element with infinite order. Assume that $M/S^1$ is orientable, then the  $L^2$-Alexander torsion is given by
	\[ \tau^{(2)}(X,\phi,\gamma) \doteq \max \left\{1, t^{x_M(\phi)}\right\}.\]
\end{corollary}
\begin{proof}
	Note that the standard fiber torus $(a;b)$ admits an effective $S^1$-action such that the fibers and the orbits of the action coincide. To choose such an action for a neighborhood of a Seifert fiber is the same as giving the corresponding point in the base space a local orientation. Hence this action extends to $M$ because $M/S^1$ is orientable. Therefore Lemma \ref{l2torofs1cwc} implies 
	\[ \LAT(M,\phi,\gamma)\doteq \max \left\{1, t^{k_\phi} \right\}^{-\chi_{orb}^{S^1} (M)}=\max \left\{1, t^{-k_\phi \chi_{orb}^{S^1} (M)} \right\} .\]
	The last equality holds because $M$ is aspherical and hence $-\chi_{orb}^{S^1} (M)\geq 0$ \cite[Theorem 5.3]{Sc83}.
	By the construction of the $S^1$-action we easily see that $k_\phi$ and $\chi_{orb}^{S^1}(M)$ are the same as in Lemma \ref{thurstonnormforallSFS}. Moreover one has the relation $\max\left\{1,t^k\right\}\doteq \max\left\{1,t^{|k|}\right\}$ for any $k\in\R$ because of the equality $\max\left\{1,t^k\right\}=t^k\cdot \max\left\{1,t^{-k}\right\}$.
\end{proof}
\subsection{The $L^2$-Alexander torsion for Seifert fiber spaces without an effective $S^1$-action}
Let $M$ be a Seifert fiber space. An embedded torus in $M$ is called $\emph{vertical}$, if it is a union of regular fibers.
One should observe that cutting $M$ along a vertical torus $T$ is exactly the same as cutting the base space of $M$ along an embedded curve which does not intersect a cone point. This will be the key observation in the next proofs.

 As indicated in the following proofs we will cut $M$ into pieces, where we can calculate the $L^2$-Alexander function. Therefore we need a gluing formula for the Thurston norm which is due to Eisenbud and Neumann \cite[Proposition 3.5]{EN85}.
\begin{theorem}\label{gluingthurstonnorom}
	Let $\mathcal{T}$ be a collection of incompressible disjoint tori embedded in $N$. Denote by $\mathcal{B}$ the collection of components of $N\setminus \mathcal{T}$ and by $i_B\colon B\rightarrow N$ the inclusion of a component $B\in\mathcal{B}$. Then the Thurston norm of each $\phi\in H^1(N;\R)$ satisfies the equality
	\[ x_N(\phi)=\sum_{B \in \mathcal{B}} x_B(i_B^{\ast}\phi).\]
\end{theorem}
\begin{lemma}\label{lem:LATKlein}
	Let $M$ be a Seifert fiber space with base space a non-orientable surface of genus $2$. Moreover let $(M,\phi,\gamma)$ be admissible, such that the image of a regular fiber under $\gamma$ is an element with infinite order. Then
	\[ \LAT({M,\phi,\gamma})\doteq\max\left\{1, t^{x_M(\phi)}\right\}\]
\end{lemma}

\begin{proof}
	We can cut $M$ along two vertical tori, such that we obtain two pieces $M_1, M_2$ both with orientable base space (see Figure \ref{kleinbottle}). Then the restriction of $\gamma$ to the tori has infinite image by hypothesis. We obtain from Proposition \ref{pro:LATgluing}, Corollary \ref{cor:Mainthm}, and Theorem \ref{gluingthurstonnorom}:
	\begin{align*}
	\LAT({M,\phi,\gamma})&\doteq \LAT({M_1,\phi i_1,\gamma i_1})\cdot \LAT({M_2,\phi i_2,\gamma i_2}) \\
	&\doteq\max \left\{1, t^{x_{M_1} (i^{\ast}_1\phi)} \right\} \cdot \max \left\{1, t^{x_{M_2}(i^{\ast}_2\phi)}\right\} \\
	&=\max \left\{1, t^{x_{M_1}(i^{\ast}_1\phi)+ x_{M_2}(i^{\ast}_2\phi)}\right\}\\
	&=\max \left\{1, t^{x_{M}(\phi) } \right\}.
	\end{align*}
	Here we used that Lemma \ref{l2torofs1cwc} yields $\LAT(T,\phi,\gamma)\doteq 1$ for a torus $T$ and $\gamma$ having infinite image.
\end{proof}
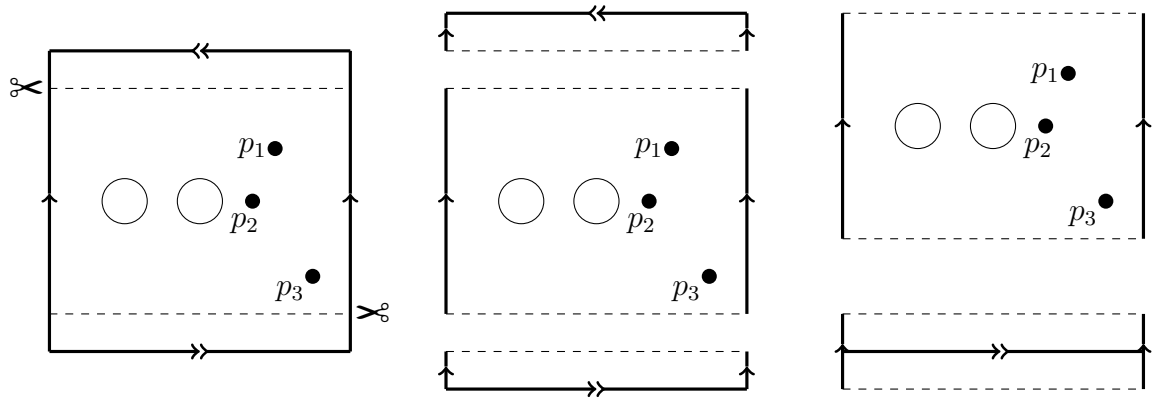
\begin{figure}
	\centering
	\begin{tikzpicture}
	\begin{scope}
	
	\node (siccorsoben) at (-0.3,3.5){\Large\LeftScissors};
	\node (siccorsunten) at (4.3,0.5){\Large\RightScissors};
	
	\draw (1,2) circle (0.3);
	\draw (2,2) circle (0.3);
	
	\fill (3,2.7) circle(0.1);
	\node (p1) at (2.7,2.7) {$p_1$};
	
	\fill (2.7,2) circle(0.1);
	\node (p2) at (2.6,1.7) {$p_2$};
	
	\fill (3.5,1) circle(0.1);
	\node (p3) at (3.2,0.8) {$p_3$};
	
	\draw[dashed] (0,0.5)--(4,0.5);
	\draw[dashed] (0,3.5)--(4,3.5);
	
	\draw[->,very thick](0,0) -- (2,0);
	\draw[>-,very thick](2,0) -- (4,0);
	
	\draw[-<,very thick](0,4) -- (2,4);
	\draw[<-,very thick](2,4) -- (4,4);
	
	\draw[->,very thick](0,0) -- (0,2.1);
	\draw[very thick](0,2.1) -- (0,4);
	
	\draw[->,very thick](4,0) -- (4,2.1);
	\draw[very thick](4,2.1) -- (4,4);
	
	\end{scope}
	
	\begin{scope}[xshift=150]
	\draw (1,2) circle (0.3);
	\draw (2,2) circle (0.3);
	
	\fill (3,2.7) circle(0.1);
	\node (p1) at (2.7,2.7) {$p_1$};
	
	\fill (2.7,2) circle(0.1);
	\node (p2) at (2.6,1.7) {$p_2$};
	
	\fill (3.5,1) circle(0.1);
	\node (p3) at (3.2,0.8) {$p_3$};
	
	\draw[dashed] (0,0.5)--(4,0.5);
	\draw[dashed] (0,3.5)--(4,3.5);
	\draw[dashed] (0,0)--(4,0);
	\draw[dashed] (0,4)--(4,4);
	
	\draw[->,very thick](0,-0.5) -- (2,-0.5);
	\draw[>-,very thick](2,-0.5) -- (4,-0.5);
	
	\draw[-<,very thick](0,4.5) -- (2,4.5);
	\draw[<-,very thick](2,4.5) -- (4,4.5);
	
	\draw[->,very thick](0,0.5) -- (0,2.1);
	\draw[very thick](0,2.1) -- (0,3.5);
	\draw[->,very thick](0,-0.5) -- (0,-0.2);
	\draw[very thick](0,-0.2) -- (0,0);
	\draw[->,very thick](0,4) -- (0,4.3);
	\draw[very thick](0,4.3) -- (0,4.5);
	
	\draw[->,very thick](4,0.5) -- (4,2.1);
	\draw[very thick](4,2.1) -- (4,3.5);
	\draw[->,very thick](4,-0.5) -- (4,-0.2);
	\draw[very thick](4,-0.2) -- (4,0);
	\draw[->,very thick](4,4) -- (4,4.3);
	\draw[very thick](4,4.3) -- (4,4.5);
	
	\end{scope}
	\begin{scope}[xshift=300]
	\draw (1,3) circle (0.3);
	\draw (2,3) circle (0.3);
	
	\fill (3,3.7) circle(0.1);
	\node (p1) at (2.7,3.7) {$p_1$};
	
	\fill (2.7,3) circle(0.1);
	\node (p2) at (2.6,2.7) {$p_2$};
	
	\fill (3.5,2) circle(0.1);
	\node (p3) at (3.2,1.8) {$p_3$};
	
	\draw[->,very thick](0,0) -- (2.1,0);
	\draw[>-,very thick](2.1,0) -- (4,0);

	\draw[dashed] (0,1.5)--(4,1.5);
	\draw[dashed] (0,4.5)--(4,4.5);
	\draw[dashed] (0,-0.5)--(4,-0.5);
	\draw[dashed] (0,0.5)--(4,0.5);
	
	\draw[->,very thick](4,1.5) -- (4,3.1);
	\draw[very thick](4,3.1) -- (4,4.5);
	\draw[->,very thick](4,-0.5) -- (4,0.1);
	\draw[very thick](4,0.1) -- (4,0.5);
	
	\draw[->,very thick](0,1.5) -- (0,3.1);
	\draw[very thick](0,3.1) -- (0,4.5);
	\draw[->,very thick](0,-0.5) -- (0,0.1);
	\draw[very thick](0,0.1) -- (0,0.5);
	\end{scope}
	\end{tikzpicture}
	\caption{A Klein bottle with two boundary components is cut along two circles to obtain two orientable pieces. The points indicate exceptional fibers. So we do not cut through them.}\label{kleinbottle}
\end{figure}
\begin{theorem}
	Let $M$ be a Seifert fiber space with a non orientable base space other than $\R P^2$ and $(M,\phi,\gamma)$ admissible, such that the image of a regular fiber under $\gamma$ is an element with infinite order. Then
	\[ \LAT({M,\phi,\gamma})\doteq \max\left\{1, t^{x_M(\phi)}\right\}.\]
\end{theorem}

\begin{proof}
	By the classification of non-orientable surfaces, we can cut the base space along embedded curves, such that every piece is a Klein bottle with boundary or a M\"obius strip. This corresponds to cutting $M$ along vertical tori, such that every connected component has a M\"obius strip or Klein bottle as base space. As in the proof above we can use the additivity of the Thurston norm and the gluing formula along tori to prove the statement for every piece. The case of the Klein bottle has been dealt with in the Lemma \ref{lem:LATKlein}. Since the doubling of a M\"obius strip is the Klein bottle, we can use a standard doubling argument for Thurston norm and $L^2$-Alexander torsion to receive the desired result.   
\end{proof}

\end{document}